\documentclass{amsart}
\usepackage[mathscr]{eucal}
\usepackage{amssymb, amsmath,array, amscd}
\usepackage{enumerate}
\usepackage{graphicx}
\usepackage{url}
\usepackage[colorlinks,plainpages,backref]{hyperref}
\usepackage{subfigure}

\input xy
\xyoption{all}

\newtheorem{theorem}{Theorem}[section]
\newtheorem{lemma}[theorem]{Lemma}
\newtheorem{corollary}[theorem]{Corollary}
\newtheorem{proposition}[theorem]{Proposition}

\theoremstyle{definition}
\newtheorem{definition}[theorem]{Definition}
\newtheorem{example}[theorem]{Example}

\begin{document}

\title{Induced and Complete Multinets}

%
%

\author{J. Bartz}
\address{Department of Mathematics \\ Francis Marion University \\
Florence \\ SC 29506 USA} \email{ jbartz@fmarion.edu}

\dedicatory{}
\thanks{}

\begin{abstract}
Multinets are certain configurations of lines and points with multiplicities in the complex projective plane $\mathbb{P}^2$. They appear in the study of resonance and characteristic varieties of complex hyperplane arrangement complements and cohomology of Milnor fibers.  In this paper, two properties of multinets, inducibility and completeness, and the relationship between them are explored with several examples presented. Specializations of multinets plays an integral role in our findings. The main result is the classification of complete 3-nets.
\end{abstract}

\keywords{hyperplane arrangement, net, multinet, induced multinet, complete multinet, $K(\pi,1)$ arrangement, free arrangement}

\maketitle

\date{\today}


\section{Introduction}
\label{sec:1}

Multinets are certain configurations of lines and points with multiplicities in the complex projective plane $\mathbb{P}^2$. More specifically they are multi-arrangements of projective lines partitioned into three or four blocks with some additional combinatorial properties (see section \ref{sec:2}). They originally arose in the study of resonance and characteristic varieties of the complement of a complex hyperplane arrangement in \cite{FY, LY}.  Multinets have also appeared while studying the cohomology of Milnor fibers in \cite{DS}.

Very few examples of multinets with non-trivial multiplicities were known initially. It was observed in \cite{FY} that several of the earliest known examples satisfied an extra property which implied the underlying arrangements were $K(\pi,1)$-arrangements. These multinets are referred to as \emph{complete multinets}. More recently a systematic method of constructing multinets was introduced in \cite{BY} and produced a variety of new examples known as \emph{induced multinets}. Not all multinets are induced. In fact two line arrangements can support the same multinet structure while not being lattice equivalent. Such arrangements are referred to as \emph{specializations} of a given multinet. In the paper we recall some definitions and known results, give examples of induced and non-induced multinets, and investigate the completeness of the induced multinets presented in \cite{BY}.  Specializations of multinets plays an integral role in our findings. The main result is the classification of complete 3-nets.
  
The paper is organized as follows. In section \ref{sec:2} we recall basic definitions and relevant properties of multinets. In section \ref{sec:3} we exhibit examples of induced and non-induced multinets while exploring the notion of specializations of a multinet. Section \ref{sec:4} discusses completeness of multinets and contains the main result of the paper, the classification of complete 3-nets. Finally some open problems are listed in section \ref{sec:5}.

\section{Preliminaries}
\label{sec:2}

We recall basic definitions and relevant properties of multinets. Several well-known examples are collected. A description of monomial groups and their arrangements is given. These arrangements were used in \cite{B,BY} to systematically construct examples of multinets known as \emph{induced multinets}.  A synopsis is given of this  construction and illustrated by obtaining a $(3,4)$-net realizing $\mathbb{Z}/2\mathbb{Z}\times\mathbb{Z}/2\mathbb{Z}$ as an induced multinet. 

\subsection{Pencils of curves and multinets}

There are several equivalent ways to define multinets. Here we present them using pencils of plane curves. A \emph{pencil of plane curves} is a line in the projective space of homogeneous polynomials from $\mathbb{C}[x_1,x_2,x_3]$ of some fixed degree $d$.  Any two distinct curves of the same degree generate a pencil, and conversely a pencil is determined by any two of its curves $C_1, C_2$. An arbitrary curve $C$ in the pencil (called a \emph{fiber}) is $C=aC_1 + bC_2$ where $[a:b]\in \mathbb{P}^1.$  Every two fibers in a pencil intersect in the same set of points $\mathcal{X}=C_1\cap C_2,$ called the \emph{base} of the pencil. If fibers do not have a common component (called a \emph{fixed component}), then the base is a finite set of points.

A curve of the form  $\prod_{i=1}^q \alpha_i^{m_i},$ where $\alpha_i$  are distinct linear forms and $m_i\in \mathbb{Z}_{>0}$ for $1\leq i\leq q,$ is called \emph{completely reducible}. Such a curve is called \emph{reduced} if $m_i=1$ for each $i$.  We are interested in connected pencils of plane curves without fixed components and at least three completely reducible fibers. By connectivity we mean the nonexistence of a completely reducible fiber whose distinct components intersect only at $\mathcal{X}$. For conciseness we refer to such a pencil as a \emph{Ceva pencil}.

\begin{definition} 
The union of all completely reducible fibers (with a fixed partition into fibers, also called \emph{blocks}) of a Ceva pencil of degree $d$ is called a  ($k,d$)-\emph{multinet} where $k$ is the number of the blocks.  The base $\mathcal{X}$ of the pencil is determined by the multinet structure and called the \emph{base} of the multinet.

If the intersection of each two fibers is transversal, i.e. $|\mathcal{X}|=d^2$ and hence all blocks are reduced, then the multinet is called a \emph{net}. If $|\mathcal{X}|<d^2$ we call the multinet \emph{proper}. If all blocks are reduced the multinet is said to be \emph{light}. If there are non-reduced blocks we call the multinet \emph{heavy}.  A block of a multinet is said to be a \emph{pencil} if all of its lines intersect at a common point.
\end{definition}

From a projective geometry perspective, a $(k,d)$-multinet is a multi-arrangement  $\mathcal{A}$ of lines in $\mathbb{P}^2$ provided with multiplicities $m(\ell)\in \mathbb{Z}_{>0}$ ($\ell\in \mathcal{A}$) and partitioned into $k$ blocks $\mathcal{A}_1,\ldots, \mathcal{A}_k$ ($k\geq 3$) subject to the following two conditions. 

(i) Let $\mathcal{X}$ be the set of the intersections of lines from different blocks. For each point $p\in\mathcal{X}$, the number $$n_p=\sum_{\ell\in \mathcal{A}_i, p\in \ell} m(\ell)$$ is independent on $i$. This number is called the \emph{multiplicity} of $p$. 

(ii) For every two lines $\ell$ and $\ell'$ from the same block, there exists a sequence of lines from that block $\ell=\ell_0,\ell_1,\ldots,\ell_r=\ell'$ such that $\ell_{i-1}\cap\ell_i\not\in\mathcal{X}$ for $1\leq i\leq r$. This is the connectivity condition.

Multinets can be defined purely combinatorially using incidence relations. Note that  the multiplicity $m(\ell)$  for each $\ell\in \mathcal{A}$ equals the multiplicity of its corresponding linear factor in the completely reducible fibers of the Ceva pencil. From this viewpoint a net is a multinet with $m(\ell)=n_p=1$ for all $\ell\in\mathcal{A}$ and $p\in\mathcal{X}$. 

From a combinatorial viewpoint $(k,d)$-nets are the realization of $k-2$ pairwise orthogonal Latin squares of size $d$ (after identifying all blocks). If $k=3$, the Latin square gives a multiplication table of a quasi-group $G$ and the associated net is said to \emph{realize $G$}. The classification of groups which can be realized by nets has been completed in \cite{KNP1, KNP, Ynet}.

\subsection{Properties of multinets and examples}

Several important properties of multinets are listed below which have been collected from \cite{FY,St, Yu2}. 

\begin{proposition} 
\label{properties}
Let $\mathcal{A}$ be a $(k,d)$-multinet. Then:
\begin{enumerate}
\item $\sum_{\ell\in \mathcal{A}_i}m(\ell)=d$, independent of $i$;
\item $\sum_{\ell\in \mathcal{A}}m(\ell)=dk$;
\item $\sum_{p\in\mathcal{X}}n_p^2=d^2$ (B\'ezout's theorem);
\item $\sum_{p\in\mathcal{X}\cap\ell}n_p=d$ for every $\ell\in \mathcal{A}$;
\item There are no multinets with $k\geq 5$; 
\item All multinets with $k=4$ are nets.
\end{enumerate}
\end{proposition}

\begin{example}
\label{trivial}
An arrangement comprised of $k$ lines which intersect at a common point supports a $(k,1)$-net  where each block consists of one line. This case corresponds to a so-called local resonance component. It is considered to be trivial and we will often tacitly assume that $d>1$.
\end{example}

\begin{example}
\label{Zn}
For each $n\geq 2$, there is a Ceva pencil generated by $x^n-y^n$ and $y^n-z^n$ with third completely reducible fiber given by $x^n-z^n$. It is commonly referred to as the Fermat pencil. The corresponding arrangement supports a $(3,n)$-net realizing $\mathbb{Z}/n\mathbb{Z}$ with each block being a pencil. The common intersection point of each block lies outside of the base of the net. For $n=3$, this is one of the specializations of a Pappus arrangement (cf. subsection \ref{(3,3)-nets}). Yuzvinsky showed in Proposition 3.3 of \cite{Ynet} that a $(3,n)$-net with all blocks being pencils is projectively equivalent to the arrangement defined by $Q=[x^n-y^n][x^n-z^n][y^n-z^n]$. 
\end{example}

\begin{example} 
\label{G(n,1,3)}
For each $n\geq1$, a $(3,2n)$-multinet is given by the pencil generated by polynomials $x^n(y^n-z^n)$ and $y^n(x^n-z^n)$ with the third completely reducible fiber being $z^n(x^n-y^n)$. These are the projectivizations of the reflection arrangements for the full monomial groups $G(n,1,3)$ (see subsection \ref{monomial}). For $n=1$, it gives the only (up to projective isomorphism) $(3,2)$-net of Coxeter type $A_3$; for $n=2$, it is the $(3,4)$-multinet of Coxeter type $B_3$. These multinets are heavy when $n>1$.
\end{example}

\begin{example}
\label{H}
The cubics $xyz$ and $x^3+y^3+z^3$ generate a Ceva pencil  with 4 completely reducible fibers. They give a $(4,3)$-net known as the Hesse configuration. This is the only currently known multinet with 4 blocks. A long-standing conjecture posed by Yuzvinsky is that the Hesse configuration is the unique 4-net up to projective isomorphism. Dunn, Miller, Wakefield, and Zwicknagl proved in \cite{DMWZ} that the Hesse configuration is the unique (4,3)-net in $\mathbb{P}^2$ and no $(4,d)$-nets exists in $\mathbb{P}^2$ for $d=4, 5, 6$.  An alternate proof of the non-existence of (4,4)-nets in $\mathbb{P}^2$ using tropical geometry was given by G\"unt\"urk\"un and K\.i\c{s}\.isel in \cite{GK}. 
\end{example}

\subsection{Monomial groups and their arrangements}
\label{monomial}

Examples of multinets can be derived from certain reflection arrangements of finite complex monomial groups. Historically monomial groups played a useful role in the representation theory of groups (see \cite{O}). Reflection arrangements of finite complex monomial groups and their subarrangements have been studied extensively, especially regarding their connections with free arrangements (see \cite{OT}). Below we summarize the description of finite complex monomial groups and their arrangements given in \cite{OT}. 

Let $V=\mathbb{C}^\ell$ with $\ell>1$ and choose a basis $\{e_1,\dots, e_\ell\}$ of $V$. For any integer $n>1$, let $\mathbb{Z}/n\mathbb{Z}$ denote the cyclic group of order $n$ generated by $\theta=\exp(2\pi i/n)$. Put $I=\{1,2,\dots, \ell\}$ and let $\epsilon: I \rightarrow \mathbb{Z}/n\mathbb{Z}$ be any function. The \emph{full monomial group}, denoted $G(n,1,\ell)$, is the subgroup of $GL(V)$ consisting of all transformations $$g(\sigma,\epsilon)e_i=\varepsilon(i)e_{\sigma(i)}$$ where $\sigma\in\textrm{Sym}(\ell)$, the symmetric group on $\ell$ symbols. Equivalently, the full monomial group is the wreath product of $\mathbb{Z}/n\mathbb{Z}$ and $\textrm{Sym}(\ell)$, consisting of all $\ell\times \ell$ monomial matrices with entries in $\mathbb{Z}/n\mathbb{Z}$. Its reflection arrangement is said to be \emph{of type $G(n,1,\ell)$} and given by $$Q=x_1\dots x_\ell \prod_{1\leq i<j\leq \ell} (x_i^n-x_j^n).$$

Another family of reflection arrangements can be defined using certain irreducible subgroups of the full monomial group $G(n,1,\ell)$. Let $p$ be divisor of $n$ and $G(n,p,\ell)$ be the subgroup of the full monomial group consisting of all $g(\sigma,\epsilon)$ where $\prod \epsilon(i)$ is a power of $\theta^p$. These groups are generated by reflections and irreducible since $n>1$. If $p<n$, the subgroup $G(n,p,\ell)$ contains the reflections $e_i\mapsto \theta^p e_i$ and $e_i\mapsto e_j$ for $i\ne j$. In this case, the corresponding reflection arrangement is the same as the reflection arrangement of the full monomial group $G(n,1,\ell)$. On the other hand, if $p=n$, the reflection arrangement of the subgroup $G(n,n,\ell)$ is defined by  $$Q=\prod_{1\leq i<j\leq \ell}(x_i^n-x_j^n)$$ and said to be of \emph{type $G(n,n,\ell)$}. For $n=2$ the full monomial group $G(2,1,\ell)$ is the Coxeter group of type $B_\ell$ and $G(2,2,\ell)$ is the Coxeter group of type $D_\ell$.

\subsection{Induced multinets}
\label{induced}

Few examples of multinets with non-trivial multiplicities were known initially. Then a systematic method of constructing multinets was introduced in \cite{B,BY} and produced a variety of new examples known as \emph{induced multinets} which we define below.  We briefly describe the method of producing induced multinets and give a summary of their combinatorial properties.

The notion of multinets can be generalized to ${\mathbb P^r}$ ($r>2$) by using pencils of homogeneous polynomials of $r+1$ variables. Presently the only known multinets in ${\mathbb P^r}$ for $r>2$ are the $(3,2n)$-nets in ${\mathbb P^3}$ given for each $n\in \mathbb{Z}_{>0}$ by the defining polynomial $$ Q_n=[(x_0^n-x_1^n)(x_2^n-x_3^n)][(x_0^n-x_2^n)(x_1^n-x_3^n)][(x_0^n-x_3^n)(x_1^n-x_2^n)]$$ where the brackets determine the blocks. This arrangement is the collection of all (projectivizations of) reflection hyperplanes of the finite complex monomial group $G(n,n,4)$ (see subsection \ref{monomial}). For $n=2$ it is the Coxeter group of type $D_4$. 

Each block of $Q_n$ is partitioned in two \emph{half-blocks} (determined by parentheses) of degree $n$ each. Notice that all the planes of a half-block intersect at one line, called the \emph{base of the half-block}. For instance the base of the leftmost half-block is given by the system $x_0=0$, $x_1=0$.

Multinets can be constructed as follows. Intersect  $Q_n$ with a plane $H$ that does not belong to $Q_n$. The resulting multi-arrangement in $H$ is denoted by $\mathcal{A}^H$  and referred to as the {\it arrangement induced by $Q_n$}. The pencil in $\mathbb P^3$ corresponding to $Q_n$ induces a pencil in ${\mathbb P^2}$ with 3 completely reducible fibers. It may happen that  the pencil has a fixed component. In this case, we cancel the fixed components obtaining a smaller arrangement $\mathcal{A}^H_0$ with a multinet structure. Abusing the notation slightly we will call $\mathcal{A}^H$ (if there is no fixed component) or $\mathcal{A}^H_0$, provided with the partitions into fibers of the induced pencil, the \emph{induced multinet}.

A systematic study of the possible combinatorics of induced multinets obtained from $Q_n$ was performed in \cite{B, BY}. Induced multinets from $Q_1$ are either $(3,2)$-nets realizing $\mathbb{Z}/2\mathbb{Z}$ or trivial (cf. Example \ref{trivial}). The following theorem from \cite{BY} gives a summary of the possibilities for $n>1$.   The first five cases are heavy multinets whereas the last five cases are light multinets. 

\begin{theorem}
\label{classify}
There are 10 possibilities for the combinatorics for induced multinets from $Q_n$. Each possibility exists and is described below. 
\begin{enumerate}
\item If $n>1$, a heavy $(3,2n)$-multinet can have three lines of multiplicity $n$ and remaining lines of multiplicity 1. This is projectively equivalent to the multinets realizing $G(n,1,3)$ discussed in Example \ref{G(n,1,3)}. For $n=2$ this is the $(3,4)$-multinet of Coxeter type $B_3$. \\
\item If $n>1$, a heavy $(3,2n)$-multinet can have a unique line of multiplicity $n$ and all other lines of multiplicity 1. The base $\mathcal{X}$ consists of two points of multiplicity $n$ and all remaining points with multiplicity 1. \\
\item  If $n>1$ is even, a heavy $(3,2n)$-multinet can have three lines of multiplicity 2 and remaining lines of multiplicity 1. The base $\mathcal{X}$ consists of $3n-3$ points of multiplicity 2 and all remaining points of multiplicity 1. For $n=2$ this is the $(3,4)$-multinet of Coxeter type $B_3$. \\
\item If $n>1$ is odd, a heavy $(3,2n)$-multinet can have two lines of multiplicity 2 and remaining lines of multiplicity 1.  The base $\mathcal{X}$ consists of $2n-1$ points of multiplicity 2 and all remaining points of multiplicity 1. \\
\item If $n>1$, a heavy $(3,2n)$-multinet can have an unique line of multiplicity 2 and all other lines of multiplicity 1.  The base $\mathcal{X}$ consists of $n$ points of multiplicity 2 and all other points of multiplicity 1.  \\
\item A light $(3,2n)$-multinet can have a unique point in $\mathcal{X}$ of multiplicity $n$.  All other points in $\mathcal{X}$ have multiplicity 1. For $n=2$ this is the $(3,2)$-net of Coxeter type $A_3$.\\
\item If $n>1$, a light $(3,2n-1)$-multinet can have a unique point in $\mathcal{X}$ of multiplicity $n-1$.  All other points in $\mathcal{X}$ have multiplicity 1. For $n=2$ this gives a $(3,3)$-net realizing $\mathbb{Z}/3\mathbb{Z}$ with each block in general position. \\
\item If $n>2$, a light $(3,2n-2)$-multinet can have a unique point in $\mathcal{X}$ of multiplicity $n-2$.  All other points in the base locus have multiplicity 1. For $n=3$ this gives a $(3,4)$-net realizing $\mathbb{Z}/2\mathbb{Z}\times \mathbb{Z}/2\mathbb{Z}$ with each block having exactly three concurrent lines and fourth line in general position (cf. Example \ref{Z2Z2}). \\
\item If $n>1$, a light $(3,2n)$-multinet can have several points of multiplicity 2 if it does not have points of multiplicity greater than $2$. The number of these points is bounded independently of $n$ by $2^{96}$. \\
\item A light $(3,2n)$-multinet can be a net which realizes the dihedral group of order $2n$.
\end{enumerate}
\end{theorem} 

\begin{example}
\label{Z2Z2}
Intersect $Q_3$ by the hyperplane $H$ defined by $x_0=(\xi +1)x_1-\xi x_2$ where $\xi$ is a primitive 3rd root of unity.  Then $\mathcal{A}^H$ has two common factors $x_1-x_2$ and $x_1-\xi x_2$. Canceling results in a $(3,4)$-multinet realizing $\mathbb{Z}/2\mathbb{Z}\times \mathbb{Z}/2\mathbb{Z}$. Using a convenient choice of labels for the lines, the three blocks are $\mathcal{A}_1=\{\ell_{11}, \ell_{12}, \ell_{13}, \ell_{14}\}$, $\mathcal{A}_2=\{\ell_{21}, \ell_{22}, \ell_{23}, \ell_{24}\}$, and $\mathcal{A}_3=\{\ell_{31}, \ell_{32}, \ell_{33}, \ell_{34}\}$ where the equations for the lines in $(\mathbb{P}^2)^*$ are 
\begin{displaymath}
\begin{array}{lllll}
\ell_{11}=[0:1:-1] & \phantom{} \qquad& \ell_{21}=[1:0:-1] & \phantom{}\qquad & \ell_{31}=[1:\xi^2:\xi] \\ \ell_{12}=[2\xi:1:0] & \phantom{} & \ell_{22}=[1:0:-\xi] & \phantom{} & \ell_{32}=[1:\xi^2:1] \\ \ell_{13}=[0:1:-\xi] & \phantom{} & \ell_{23}=[1:0:-\xi^2] & \phantom{} & \ell_{33}=[1:-\xi^2:0] \\
\ell_{14}=[0:1:-\xi^2] & \phantom{} & \ell_{24}=[\xi:2:0] & \phantom{} & \ell_{34}=[\xi:1:1]. \\
\end{array}
\end{displaymath}
The base locus consists of sixteen points of multiplicity 1, namely
\begin{displaymath}
\begin{array}{lll}
\ell_{11}\cap\ell_{21}\cap\ell_{31}=[1:1:1] & \phantom{}\qquad & \ell_{13}\cap\ell_{21}\cap\ell_{33}=[1:\xi:1] \\
\ell_{11}\cap\ell_{22}\cap\ell_{32}=[\xi:1:1] & \phantom{} & \ell_{13}\cap\ell_{22}\cap\ell_{34}=[\xi:\xi:1] \\
\ell_{11}\cap\ell_{23}\cap\ell_{33}=[\xi^2:1:1] & \phantom{} & \ell_{13}\cap\ell_{23}\cap\ell_{31}=[\xi^2:\xi:1] \\
\ell_{11}\cap\ell_{24}\cap\ell_{34}=[-2\xi^2:1:1] & \phantom{} &  \ell_{13}\cap\ell_{24}\cap\ell_{32}=[-2:\xi:1] \\
\ell_{12}\cap\ell_{21}\cap\ell_{32}=[1:-2\xi:1] & \phantom{} & \ell_{14}\cap\ell_{21}\cap\ell_{34}=[1:\xi^2:1] \\
\ell_{12}\cap\ell_{22}\cap\ell_{31}=[\xi:-2\xi^2:1] & \phantom{} & \ell_{14}\cap\ell_{22}\cap\ell_{33}=[\xi:\xi^2:1] \\
\ell_{12}\cap\ell_{23}\cap\ell_{34}=[\xi^2:-2:1] & \phantom{} & \ell_{14}\cap\ell_{23}\cap\ell_{32}=[\xi^2:\xi^2:1] \\
\ell_{12}\cap\ell_{24}\cap\ell_{33}=[0:0:1] & \phantom{} & \ell_{14}\cap\ell_{24}\cap\ell_{31}=[-2\xi^2:\xi^2:1]. \\
\end{array}
\end{displaymath}
Using this choice of labels, we can see that this net realizes $\mathbb{Z}/2\mathbb{Z}\times\mathbb{Z}/2\mathbb{Z}$ by observing that the point $\ell_{1i}\cap\ell_{2j}\cap\ell_{3k}\in \mathcal{X}$ appears in the associated Latin square $$\left[\begin{array}{cccc} 1 & 2 & 3 & 4  \\ 2 & 1 & 4 & 3 \\ 3 & 4 & 1 & 2 \\ 4 & 3 & 2 & 1 \end{array}\right]$$as $k$ in the $(i,j)$-th position. Note that each block has exactly three concurrent lines and a fourth line in general position.
\end{example}

\section{Specializations of Multinets and Inducibility from $Q_n$}
\label{sec:3}

Although induced multinets from $Q_n$ provide a wealth of examples of multinets, it is known that not all multinets can be obtained in this manner. For instance, every $(3,2n+1)$-net for $n\in \mathbb{Z}_{>1}$ is not an induced multinet.  An example of a proper multinet which is not induced is given in Problem 4 of \cite{BY}.  Unlike the examples for nets, there is a specialization of this latter example which is an induced multinet from $Q_n$. To be precise, a {\it specialization} of a multinet is any line arrangement in $\mathbb{P}^2$ which satisfies the incidence relations of the given multinet. A multinet may have more than one specialization. That is, there may be line arrangements with nonisomorphic intersection lattices which satisfy the same multinet incidence relations. 

In \cite{St1} Stipins constructs $(k,d)$-nets from $k-2$ mutually orthogonal Latin squares of order $d$. These Latin squares contain the combinatorial data regarding the incidence relations of the associated nets and are used to generate a parameterized family of explicit defining equations. In this section, we explore the possible specializations of several multinets using the results of Stipins and incidence relations of the associated multinets.

\subsection{(3,2)-nets}

 Any $(3,2)$-net in $\mathbb{P}^2$ is projectively equivalent to the arrangement with defining polynomial $[x(y-z)][y(x-z)][z(x-y)]$.  In particular it consists of classes which are all in general position (which coincides with being pencil for $d=2$) and is associated with the Latin square $$\left[\begin{array}{cc} 1 & 2  \\ 2 &  1 \end{array}\right].$$ Thus all $(3,2)$-nets are lattice equivalent, realize $\mathbb{Z}/2\mathbb{Z}$, and have Coxeter type $A_3$. These nets can be induced from $Q_1$. Also, these arrangements are simplicial and denoted $\mathcal{A}(6,1)$ in Gr\"unbaum's catalogue of simplicial arrangements in \cite{G}.

\subsection{(3,3)-nets}
\label{(3,3)-nets}

Up to isotopy, there is a unique Latin square of order 3, namely $$\left[\begin{array}{ccc} 1 & 2 & 3 \\ 2 & 3 & 1 \\ 3 & 1 & 2 \end{array}\right].$$ This is the multiplication table for $\mathbb{Z}/3\mathbb{Z}.$ Assuming that the block $\mathcal{A}_1=\{\ell_{11}, \ell_{12}, \ell_{13}\}$ is in general position, Stipins derived the family of line arrangements in $\mathbb{P}^2$ indexed by $[s_0:s_1]\times [t_0:t_1]\in \mathbb{P}\times \mathbb{P}$ given by 
\begin{displaymath}
\begin{array}{lllll}
\ell_{11}=[1:0:0] & \phantom{}\qquad & \ell_{21}=[1:1:1] & \phantom{}\qquad & \ell_{31}=[s_0:s_1:s_1] \\ \ell_{12}=[0:1:0] & \phantom{} & \ell_{22}=[s_0t_1:s_1t_1:s_1t_0] & \phantom{} & \ell_{32}=[t_0:t_1:t_0] \\ \ell_{13}=[0:0:1] & \phantom{} & \ell_{23}=[s_0t_0:s_0t_1:s_1t_0] & \phantom{} & \ell_{33}=[s_0t_1:s_0t_1:s_1t_0]. \\
\end{array}
\end{displaymath}
This line arrangement is a $(3,3)$-net realizing $\mathbb{Z}/3\mathbb{Z}$ for generic indices with the other blocks given by $\mathcal{A}_2=\{\ell_{21}, \ell_{22}, \ell_{23}\}$ and $\mathcal{A}_3=\{\ell_{31}, \ell_{32}, \ell_{33}\}$.  Note that the nine lines are distinct if and only if $s_0,s_1,t_0,t_1\ne 0$, $s_0\ne s_1$, $t_0\ne t_1$, and  $s_0/s_1\ne t_0/t_1$ . This family can be reindexed by the parameters $\lambda$ and $\mu$ by normalizing the original indices.  That is, put $[s_0:s_1]=[1:s_1/s_0]=[1:\lambda]$ and $[t_0:t_1]=[1:t_1/t_0]=[1:\mu]$ where $\lambda, \mu \ne 0,1$ and $\lambda\ne\mu$. In terms of this reparameterization, the family of line arrangements can be written in $(\mathbb{P}^2)^*$ as
\begin{displaymath}
\begin{array}{lllll}
\ell_{11}=[1:0:0] & \phantom{}\qquad & \ell_{21}=[1:1:1] & \phantom{}\qquad & \ell_{31}=[1:\lambda:\lambda] \\ \ell_{12}=[0:1:0] & \phantom{} & \ell_{22}=[\mu:\lambda \mu:\lambda] & \phantom{} & \ell_{32}=[1:\mu:1] \\ \ell_{13}=[0:0:1] & \phantom{} & \ell_{23}=[1:\mu:\lambda] & \phantom{} & \ell_{33}=[\mu:\mu:\lambda]. \\
\end{array}
\end{displaymath}
Denoting points $\ell_{1i}\cap\ell_{2j}\cap\ell_{3k}\in\mathcal{X}$ as the triple $(i,j,k)$, the nine points of $\mathcal{X}$ are
\begin{displaymath}
\begin{array}{lllll}
(1,1,1)=[0:1:-1] & \phantom{}\qquad & (1,2,2)=[0:1:-\mu] & \phantom{}\qquad & (1,3,3)=[0:\lambda:-\mu] \\ (2,1,2)=[1:0:-1] & \phantom{} & (2,2,3)=[\lambda:0:-\mu] & \phantom{} & (2,3,1)=[\lambda:0:-1] \\ (3,1,3)=[1:-1:0] & \phantom{} & (3,2,1)=[\lambda:-1:0] & \phantom{} & (3,3,2)=[\mu:-1:0]. \\
\end{array}
\end{displaymath}
\noindent The intersections points within the three blocks of this family are 
$$\begin{array}{l}
\ell_{11}\cap\ell_{12} = [0:0:1] \\
\ell_{11}\cap\ell_{13}= [0:1:0]  \\
\ell_{12}\cap \ell_{13}=[1:0:0] \\
\ell_{21}\cap \ell_{22}= [\lambda(1-\mu):\mu-\lambda:\mu(\lambda-1)]   \\
\ell_{21}\cap\ell_{23} = [\lambda-\mu:1-\lambda:\mu-1)]  \\   
\ell_{22}\cap \ell_{23} = [\lambda\mu(\lambda-1):\lambda(1-\mu):\mu(\mu-\lambda)] \\
\ell_{31}\cap\ell_{32}=[\lambda(1-\mu):\lambda-1:\mu-\lambda] \\
\ell_{31}\cap \ell_{33}= [\lambda(\lambda-\mu):\lambda(\mu-1):\mu(1-\lambda)] \\
\ell_{32}\cap\ell_{33} = [\mu(\lambda-1):\mu-\lambda:\mu(1-\mu)].
\end{array}$$
\noindent Each block consists of three lines which can either be a pencil or in general position. Recall it was assumed that the block $\mathcal{A}_1$ is in general position. Direct computations show that the block $\mathcal{A}_2$ is a pencil if and only if $\mu\lambda^2-3\mu \lambda + \lambda+\mu^2=0$. Similarly, $\mathcal{A}_3$ is a pencil if and only if $\lambda \mu^2-3\lambda \mu +\mu+\lambda^2=0.$ If $\mathcal{A}_2$ and $\mathcal{A}_3$ are both pencils, solving the pair of equations for $\mu$ gives $\mu=\lambda$ or $\mu=1$.  However these values for $\mu$ do not define a $(3,3)$-net as noted above. On the other hand, given a generic value for $\lambda$, the equation associated with $\mathcal{A}_2$ can be solved for $\mu$ to find parameters so that $\mathcal{A}_2$ is a pencil and $\mathcal{A}_3$ is in general position (see Figure \ref{fig:1a}). For generic $\lambda$ and $\mu$, both equations will not be satisfied and yields a $(3,3)$-net with all blocks in general position (see Figure \ref{fig:1b}).

\begin{figure}
\centering
\subfigure[Two blocks in general position] 
{
\label{fig:1a}
\includegraphics*[width=2.2in]{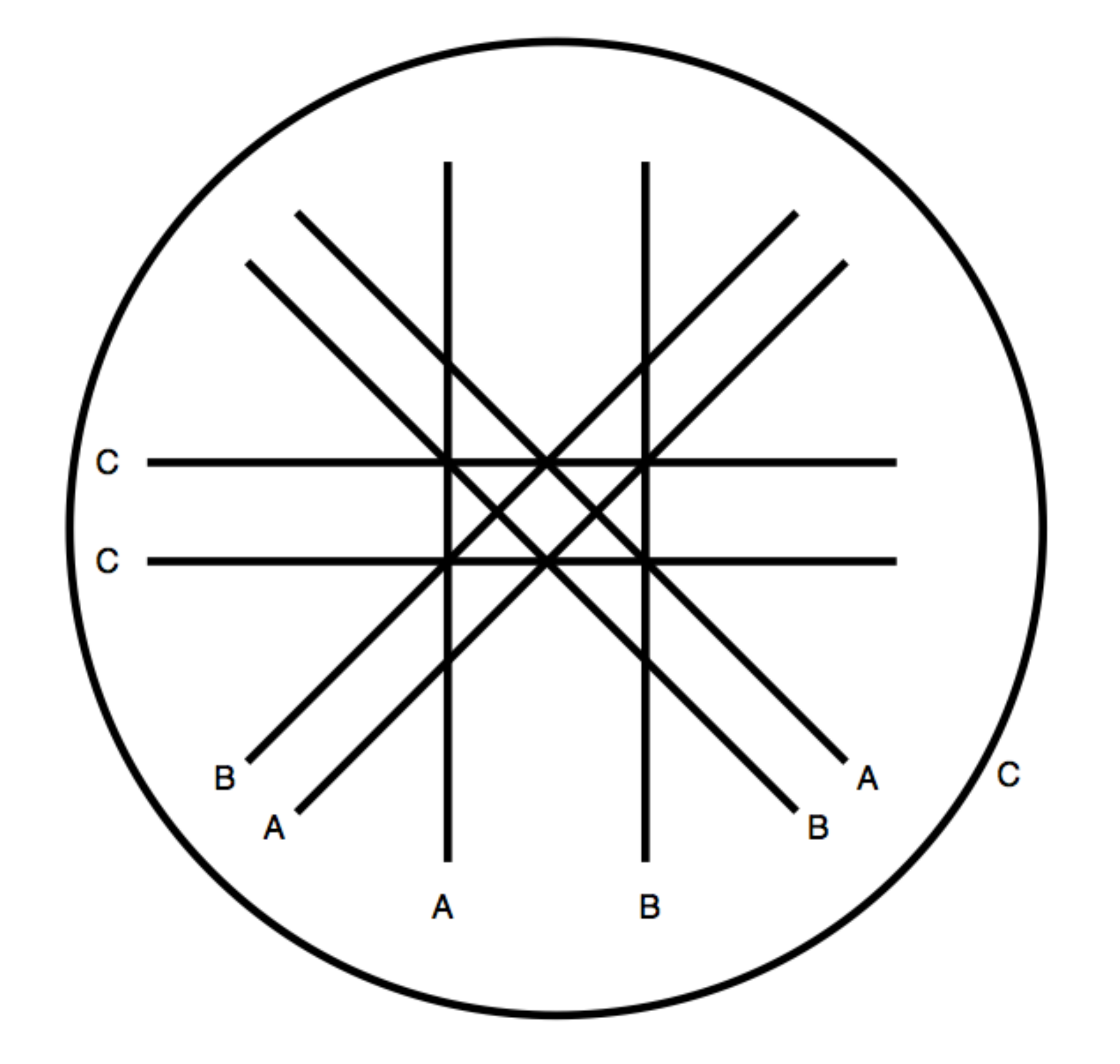}
}
\subfigure[Three blocks in general position]
{
\label{fig:1b}
\includegraphics*[width=2.2in]{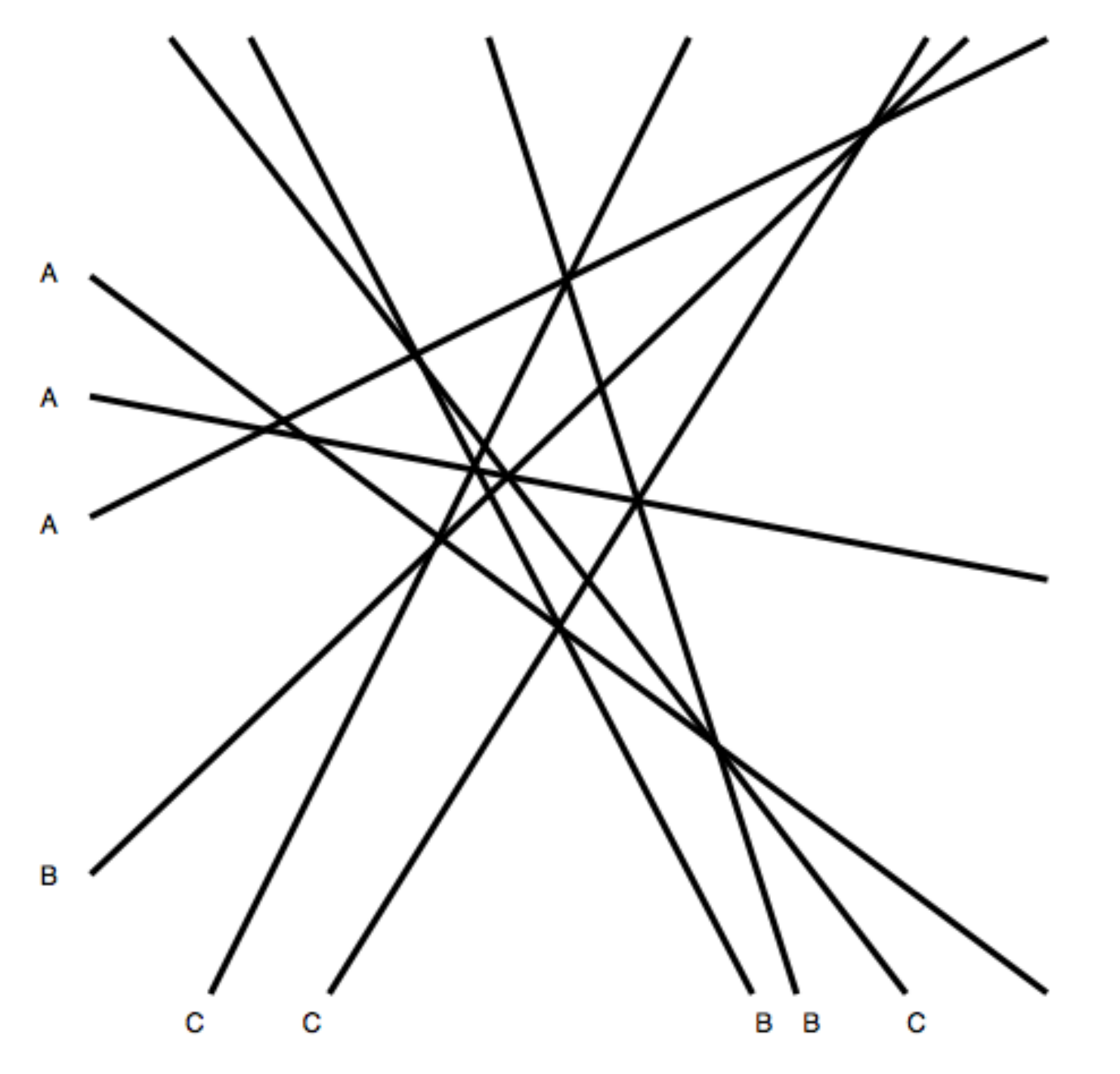}
}
\caption{Equivalent $(3,3)$-nets with different class structures}
\label{fig:1}
\end{figure}

The remaining possible specialization of the $(3,3)$-net realizing $\mathbb{Z}/3\mathbb{Z}$ consists of all blocks being pencils. This configuration is possible and was discussed in Example \ref{Zn}. We summarize our findings in the following result.

\begin{theorem}
Any $(3,3)$-net realizes $\mathbb{Z}/3\mathbb{Z}$ and has exactly one of the following block structures: \\
1. every block is in general position; \\
2. one block is a pencil, two blocks are in general position; \\
3. every block is a pencil.
\end{theorem}

In particular, there does not exist a $(3,3)$-net in $\mathbb{P}^2$ with one block in general position and two blocks being pencils. The specialization with all blocks in general position can be induced from $Q_2$ using cancellation.  The other cases are not inducible from $Q_n$. However, the specialization with all classes being pencils appears as a subarrangement of the $(3,6)$-net realizing the dihedral group of order 6 induced from $Q_3$.

\subsection{$(3,2n)$-multinets of type $G(n,1,3)$} 

For each $n\geq 1$, the arrangement defined by $[x^n(y^n-z^n)][y^n(x^n-z^n)][z^n(x^n-y^n)]$ supports a $(3,2n)$-multinet (see example \ref{G(n,1,3)}). This multinet is the projectivization of the reflection arrangement of the monomial group of type $G(n,1,3)$ and thus referred to as being \emph{of type $G(n,1,3)$}. We show below that any specialization of this multinet is projectively equivalent to the arrangement with defining polynomial  $[x^n(y^n-z^n)][y^n(x^n-z^n)][z^n(x^n-y^n)]$. Thus any $(3,2n)$-multinets of type $G(n,1,3)$ are lattice equivalent.

Any specialization of a $(3,2n)$-multinet of type $G(n,1,3)$ consists of three lines of multiplicity $n$ and $3n$ lines of multiplicity 1. The base $\mathcal{X}$ has three points of multiplicity $n$ and $n^2$ points of multiplicity 1. It contains the $(3,n)$-net realizing $\mathbb{Z}/n\mathbb{Z}$ with all blocks being pencils as a subarrangement (see example \ref{Zn}). By Proposition 3.3 of \cite{Ynet}, this subarrangement is projectively equivalent to the arrangement defined by $[x^n-y^n][x^n-z^n][y^n-z^n]$. It remains to determine the possible ways to add the remaining three lines of multiplicity $n$ to this subarrangement and obtain a multinet of type $G(n,1,3)$. 

The three points of multiplicity $n$ are the intersection of two lines of multiplicity $n$ and $n$ lines of multiplicity 1. Thus these three points are the common intersection point of each block, namely $[1:0:0]$, $[0:1:0]$ and $[0:0:1]$. The lines of multiplicity $n$ are the three lines which pass through exactly two of these points, namely $x$, $y$, and $z$. Taking multiplicity into account, we can make the following conclusion. 

\begin{theorem}
Any $(3,2n)$-multinet of type $G(n,1,3)$ is projectively equivalent to the arrangement with defining polynomial $[x^n(y^n-z^n)][y^n(x^n-z^n)][z^n(x^n-y^n)].$ 
\end{theorem}
These multinets are inducible from $Q_n$ (see subsection \ref{induced}).

\subsection{Light $(3,4)$-multinet with unique double point}
 Next we turn our attention to specializations of the light $(3,4)$-multinet with a unique point of multiplicity 2 and all other points of multiplicity 1 in the base $\mathcal{X}$. Two such examples with different block structures are presented in Figure \ref{fig:2}.  The specialization in Figure \ref{fig:2a} is the simplicial arrangement known as $\mathcal{A}(12,1)$ in Gr\"unbaum's catalogue of simplicial arrangements in \cite{G}. The specialization with two blocks in general position in Figure \ref{fig:2b} first appeared in \cite{FY}.   

\begin{figure}
\centering
\subfigure[No blocks in general position] 
{
\label{fig:2a}
\includegraphics*[width=2.1in]{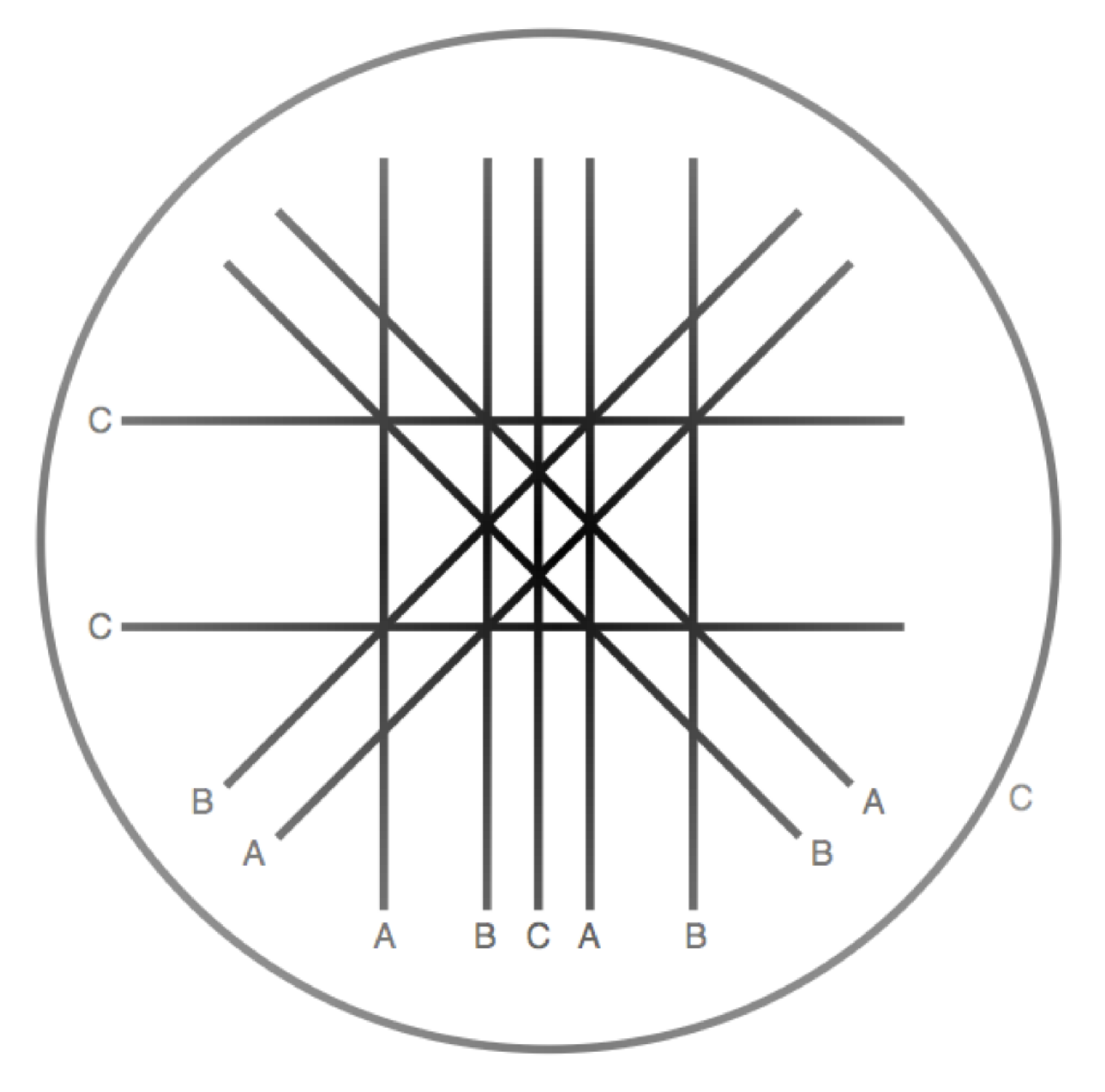}
}
\subfigure[Two blocks in general position]
{
\label{fig:2b}
\includegraphics*[width=2.25in]{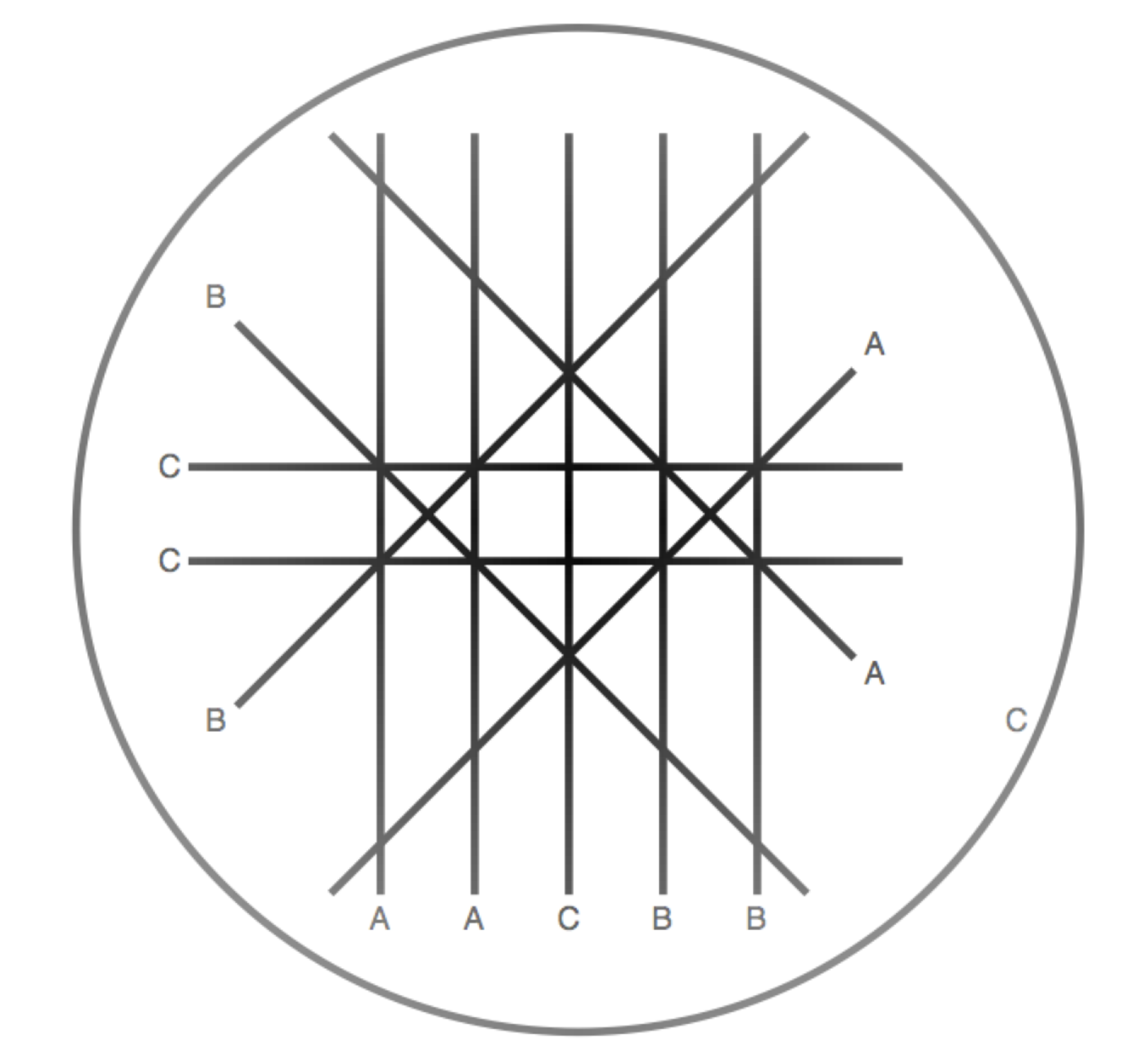}
}
\caption{Equivalent light and proper $(3,4)$-multinets}
\label{fig:2}
\end{figure}

Here each block consists of four lines. There are three possibilities for each block: (1) the four lines are a pencil; (2) the lines are in general position; or (3) exactly three of the four lines meet at a common point. A block in the latter configuration is said to be an \emph{easel} due to its resemblence to an artists' easel.  

We can eliminate the first possibility with the following lemma.

\begin{lemma}
If a light and proper multinet has a block which is a pencil, then it is the trivial multinet with $|\mathcal{X}|=1$.
\end{lemma}

\begin{proof}
Let $p$ be the common intersection point from the block which is a pencil. Then $p\in\mathcal{X}$ because the multinet is light and proper. It follows that $p$ lies on every line from the other two blocks, hence $\mathcal{X}=\{p\}$. This is the trivial multinet. \end{proof}

It follows that each block is either (1) in general position with six double points or (2) an easel with a unique triple point and 3 double points. To determine the specializations possible for this multinet, we use the multinet incidence structure to compute the realization space. 

We may assume the unique point $p\in \mathcal{X}$ with multiplicity 2 has coordinates $[0:0:1]$ in $\mathbb{P}^2$. Since no block is a pencil, we may choose coordinates on $(\mathbb{P}^2)^*$ and labels for the lines in the each block, namely $\mathcal{A}_1=\{\ell_{11}, \ell_{12}, \ell_{13}, \ell_{14}\}$, $\mathcal{A}_2=\{\ell_{21}, \ell_{22}, \ell_{23}, \ell_{24}\}$, and $\mathcal{A}_3=\{\ell_{31}, \ell_{32}, \ell_{33}, \ell_{34}\}$, so that 
\begin{displaymath}
\begin{array}{lllll}
\ell_{11}=[1:1:1] & \phantom{}\qquad & \ell_{21}=[1:0:0]  & \phantom{}\qquad & \ell_{31}=[s_0:s_1:s_2]\\
\ell_{12}=[0:1:\lambda] & \phantom{} & \ell_{22}=[0:1:0] & \phantom{} & \ell_{32}=[0:1:t] \\
\ell_{13}=[0:1:\mu] & \phantom{} & \ell_{23}=[0:0:1] & \phantom{} & \ell_{33}=[0:1:u] \\
\ell_{14}=[x_0:x_1:x_2] & \phantom{} & \ell_{24}=[y_0:y_1:y_2] & \phantom{} & \ell_{34}=[v_0:v_1:v_2].
\end{array}
\end{displaymath}
Here we assumed that $p=\ell_{12}\cap\ell_{13}\cap\ell_{22}\cap\ell_{23}\cap\ell_{32}\cap\ell_{33}$. Using the incidence relations imposed by the multinet structure, we can express this family of line arrangements in terms of the two parameters $\lambda$ and $\mu$. The blocks $\mathcal{A}_1$ and $\mathcal{A}_2$ can be used to compute the coordinates for the twelve points of $\mathcal{X}$ with multiplicity 1 as
\begin{displaymath}
\begin{array}{lll}
\ell_{11}\cap\ell_{21}=[0:1:-1] & \phantom{}\qquad & \ell_{13}\cap\ell_{21}=[0:\mu:-1] \\
\ell_{11}\cap\ell_{22}=[1:0:-1] & \phantom{} & \ell_{13}\cap\ell_{24}=[\mu y_1-y_2:-\mu y_0:y_0] \\
\ell_{11}\cap\ell_{23}=[1:-1:0] & \phantom{} & \ell_{14}\cap\ell_{21}=[0:x_2:-x_1] \\
\ell_{11}\cap\ell_{24}=[y_1-y_2:y_2-y_0:y_0-y_1] & \phantom{} &  \ell_{14}\cap\ell_{22}=[x_2:0:-x_0] \\
\ell_{12}\cap\ell_{21}=[0:\lambda:-1] & \phantom{} & \ell_{14}\cap\ell_{23}=[x_1:-x_0:0] \\
\ell_{12}\cap\ell_{24}=[\lambda y_1-y_2:-\lambda y_0:y_0] & \phantom{} & \ell_{14}\cap\ell_{24}=[z_0:z_1:z_2] \\
\end{array}
\end{displaymath}
where 
$$\begin{array}{l}
z_0=x_2y_1-x_1y_2 \\
z_1=x_0y_2-x_2y_0 \\
z_2=x_1y_0-x_0y_1.
\end{array}$$
Consider the line $\ell_{32}$ and its intersections with lines from $\mathcal{A}_1$ and $\mathcal{A}_2$. Each intersection point lies in $\mathcal{X}$. The double point $p$ is the intersection with $\ell_{12}, \ell_{13}, \ell_{22},$ and $\ell_{23}$. It follows that $\ell_{32}$ passes through either $\ell_{11}\cap\ell_{21}$ and $\ell_{14}\cap\ell_{24}$, or $\ell_{11}\cap\ell_{24}$ and $\ell_{14}\cap\ell_{21}$. We may choose our labels so that $\ell_{11}\cap\ell_{21}$ and $\ell_{14}\cap\ell_{24}$ lie on $\ell_{32}$ which implies $t=1$. It follows that $\ell_{11}\cap\ell_{24}$ and $\ell_{14}\cap\ell_{21}$ lie on $\ell_{33}$.

Next consider the line $\ell_{22}$. A similiar argument shows $\ell_{22}$ passes through either $\ell_{11}\cap\ell_{31}$ and $\ell_{14}\cap\ell_{34}$, or $\ell_{11}\cap\ell_{34}$ and $\ell_{14}\cap\ell_{31}$. We may choose our labels so that $\ell_{11}\cap\ell_{31}$ and $\ell_{14}\cap\ell_{34}$ lie on $\ell_{22}$. It follows that  $\ell_{11}\cap\ell_{34}$ and $\ell_{14}\cap\ell_{31}$ lie on $\ell_{23}$. The incidence relations for twelve points in $\mathcal{X}$ of multiplicity 1 based on these choice of labels are
\begin{displaymath}
\begin{array}{lllll}
\ell_{11}\cap\ell_{21}\cap\ell_{32}& \phantom{}\qquad & \ell_{12}\cap\ell_{21}\cap\ell_{31} & \phantom{}\qquad &  \ell_{14}\cap\ell_{21}\cap\ell_{33} \\
\ell_{11}\cap\ell_{22}\cap\ell_{31} & \phantom{} & \ell_{12}\cap\ell_{24}\cap\ell_{34}  & \phantom{} & \ell_{14}\cap\ell_{22}\cap\ell_{34}  \\
\ell_{11}\cap\ell_{23}\cap\ell_{34}& \phantom{} & \ell_{13}\cap\ell_{21}\cap\ell_{34} & \phantom{} &  \ell_{14}\cap\ell_{23}\cap\ell_{31} \\
\ell_{11}\cap\ell_{24}\cap\ell_{33} & \phantom{} & \ell_{13}\cap\ell_{24}\cap\ell_{31}  & \phantom{} &\ell_{14}\cap\ell_{24}\cap\ell_{32}.\\
\end{array}
\end{displaymath}
It is straightforward to use these incidence relations to express this family in terms of the parameters $\lambda$ and $\mu$, namely
\begin{displaymath}
\begin{array}{lllll}
\ell_{11}=[1:1:1] & \phantom{}\qquad & \ell_{21}=[1:0:0] & \phantom{}\qquad & \ell_{31}=[\lambda:1:\lambda] \\ \ell_{12}=[0:1:\lambda] & \phantom{} & \ell_{22}=[0:1:0] & \phantom{} & \ell_{32}=[0:1:1] \\ \ell_{13}=[0:1:\mu] & \phantom{} & \ell_{23}=[0:0:1] & \phantom{} & \ell_{33}=[1:1:\lambda\mu] \\
\ell_{14}=[\lambda:1:\lambda\mu] & \phantom{} & \ell_{24}=[\lambda:1+\lambda:\lambda(1+\mu)] & \phantom{} & \ell_{34}=[1:1:\mu]. \\
\end{array}
\end{displaymath}
The base $\mathcal{X}$ consists of thirteen points. The twelve points of multiplicity 1 are
\begin{displaymath}
\begin{array}{lll}
\ell_{11}\cap\ell_{21}\cap\ell_{32}=[0:1:-1] & \phantom{}\qquad & \ell_{13}\cap\ell_{21}\cap\ell_{34}=[0:\mu:-1] \\
\ell_{11}\cap\ell_{22}\cap\ell_{31}=[1:0:-1] & \phantom{} & \ell_{13}\cap\ell_{24}\cap\ell_{31}=[\lambda-\mu:\lambda\mu:-\lambda] \\
\ell_{11}\cap\ell_{23}\cap\ell_{34}=[1:-1:0] & \phantom{} & \ell_{14}\cap\ell_{21}\cap\ell_{33}=[0:\lambda\mu:-1] \\
\ell_{11}\cap\ell_{24}\cap\ell_{33}=[1-\lambda\mu:\lambda\mu:-1] & \phantom{} &  \ell_{14}\cap\ell_{22}\cap\ell_{34}=[\mu:0:-1] \\
\ell_{12}\cap\ell_{21}\cap\ell_{31}=[0:\lambda:-1] & \phantom{} & \ell_{14}\cap\ell_{23}\cap\ell_{31}=[1:-\lambda:0] \\
\ell_{12}\cap\ell_{24}\cap\ell_{34}=[\lambda-\mu:-\lambda:1] & \phantom{} & \ell_{14}\cap\ell_{24}\cap\ell_{32}=[\lambda\mu-1:\lambda:-\mu] \\
\end{array}
\end{displaymath}
and the point of multiplicity 2 is 
$$\ell_{12}\cap\ell_{13}\cap\ell_{22}\cap\ell_{23}\cap\ell_{32}\cap\ell_{33}=[1:0:0].$$
For the twelve lines and thirteen points of $\mathcal{X}$ to all be distinct, the parameters must satisfy the following conditions: $\lambda\ne \mu$, $\lambda\mu\ne 1$, and $\lambda,\mu\ne 0,1$. Observe that the intersection points within each block are
\begin{displaymath}
\begin{array}{lll}
\ell_{11}\cap\ell_{12}=[1-\lambda:\lambda:-1] & \phantom{}\qquad & \ell_{22}\cap\ell_{23}=[1:0:0] \\
\ell_{11}\cap\ell_{13}=[1-\mu:\mu:-1] & \phantom{} &\ell_{22}\cap\ell_{24}=[\lambda+\lambda\mu:0:-\lambda] \\
\ell_{11}\cap\ell_{14}=[\lambda\mu-1:\lambda-\lambda\mu:1-\lambda] & \phantom{} & \ell_{23}\cap\ell_{24}=[1+\lambda:-\lambda:0] \\
\ell_{12}\cap\ell_{13}=[1:0:0] & \phantom{} & \ell_{31}\cap\ell_{32}=[\lambda-1:\lambda:-\lambda] \\ 
\ell_{12}\cap\ell_{14}=[1-\mu:-\lambda:1] & \phantom{} & \ell_{31}\cap\ell_{33}=[\mu-1:-\lambda\mu:1] \\
\ell_{13}\cap\ell_{14}=[\mu-\lambda\mu:-\lambda\mu:\lambda] & \phantom{} &  \ell_{31}\cap\ell_{34}=[\lambda-\mu:\lambda\mu-\lambda:1-\lambda] \\
\ell_{21}\cap\ell_{22}=[0:0:1] & \phantom{} & \ell_{32}\cap\ell_{33}=[1:0:0] \\
\ell_{21}\cap\ell_{23}=[0:1:0] & \phantom{} & \ell_{32}\cap\ell_{34}=[1-\mu:-1:1] \\
\ell_{21}\cap\ell_{24}=[0:\lambda+\lambda\mu:-\lambda-1] & \phantom{} & \ell_{33}\cap\ell_{34}=[\lambda\mu-\mu:-\lambda\mu:1].\\
\end{array}
\end{displaymath}

Recall that each block is either (1) in general position with six double points or (2) an easel with a unique triple point and 3 double points. The block $\mathcal{A}_1$ is in general position except if $\mu=2-\lambda$ or $\mu=\lambda/(2\lambda-1)$; $\mathcal{A}_2$ is in general position unless $\lambda=-1$ or $\mu=-1$; and $\mathcal{A}_3$ is in general position unless $\mu=1/(2-\lambda)$ or $\mu=(2\lambda-1)/\lambda$. All three blocks are in general position for generic choice of $\lambda$ and $\mu$. Exactly two blocks are in general position by choosing $\lambda=-1$ and a generic value of $\mu$ (see Figure \ref{fig:2b}). 

To resolve the remaining cases, choose $\lambda=-1$. Then $\mathcal{A}_1$ is in general position unless $\mu=3$ or $\mu=1/3$. Either choice of values for the parameters also satisfies the relation for $\mathcal{A}_3$, hence no blocks are in general position (see Figure \ref{fig:2a}). It is straightforward to verify that if any two blocks are easels, then the third block is also an easel. This gives us the following result.

\begin{theorem}
Any light $(3,4)$-multinet with base $\mathcal{X}$ consisting of a unique double point and all other points of multiplicity 1 has exactly one of the following blocks structures: \\
1. every block is in general position; \\
2. one block is an easel, two blocks are in general position; \\
3. every block is an easel.
\end{theorem}
The specialization with all blocks in general position can be induced from $Q_2$. It is also possible to induce the specialization with every block being an easel from $Q_3$ using double cancellation. On the other hand, it is not possible to induce from $Q_n$ the specialization with exactly two blocks in general position.

\section{Complete Multinets}
\label{sec:4}

A Riemann-Hurwitz type formula was obtained for multinets in \cite{FY} by calculating the Euler characteristic of the blowup of $\mathbb{P}^2$ at the points of $\mathcal{X}$ using the Ceva pencil. This formula can be used to determine whether all singular fibers of a Ceva pencil associated to a multinet are completely reducible. In this case, the complement of the arrangement is aspherical and the multinet is referred to as a $K(\pi,1)$-arrangement.

\subsection{Classification of complete $3$-nets} 
We recall the definition of complete multinets and results obtained in \cite{FY}. Then we present and establish our main result, namely the classification of complete $3$-nets.

We begin by introducing some additional notation. Let $P_\mathcal{A}$ be the set of intersection points of $\mathcal{A}$. Let $\overline{\mathcal{X}}$ denote the set of intersection points of $\mathcal{A}$ not contained in $\mathcal{X}$. Thus $P_{\mathcal{A}}=\overline{\mathcal{X}}\sqcup\mathcal{X}$ and $\overline{\mathcal{X}}\cap P_{\mathcal{A}_i}$ is the set of intersection points of the block $\mathcal{A}_i$ not contained in $\mathcal{X}$. For $p\in \overline{\mathcal{X}}$, let $m_p$ be the multiplicity of $p$ in $\mathcal{A}$. The next two results and subsequent definition were introduced in \cite{FY}.
\begin{theorem}
Let $\mathcal{A}$ be a $(k,d)$-multinet, and let $\pi: \mathbb{P}^2\rightarrow \mathbb{P}^1$ be the associated Ceva pencil. Then 
\begin{eqnarray}
\label{E1}
3+|\mathcal{X}| & \geq & (2-k)[3d-d^2+\sum_{p\in \mathcal{X}}(n_p^2-n_p)]+2|\mathcal{A}|-\sum_{p\in \overline{\mathcal{X}}}(m_p-1) 
\end{eqnarray}
with equality if and only if the blocks of $\mathcal{A}$ form the only singular fibers of $\pi$. 
\end{theorem}

\begin{corollary}
Equality holds in (1) if and only if the restriction of $\pi$ to the complement $M=\mathbb{P}^2-\left(\cup\mathcal{A}\right)$ of $\mathcal{A}$ is a smooth bundle projection with base $B=\mathbb{P}^1-(k \textrm{ points})$ and fiber a smooth surface with some points removed. In particular, $\mathcal{A}$ is a $K(\pi,1)$-arrangement.
\end{corollary}

\begin{definition}
\label{complete}
A $(k,d)$-multinet $(d\geq 2)$ or its associated Ceva pencil is called {\it complete} if the equality holds in  (\ref{E1}). When $k=3$ this condition reduces to
\begin{eqnarray}
\label{E2}
\sum_{p\in \overline{\mathcal{X}}} (m_p-1) &\geq& 2|\mathcal{A}|-|\mathcal{X}|-3(d+1)+\sum_{p\in \mathcal{X}} n_p.
\end{eqnarray}  
Thus the underlying arrangement of a complete multinet is a $K(\pi,1)$-arrangement.
\end{definition}

Falk and Yuzvinsky present several examples of complete multinets in \cite{FY}, specifically the arrangements presented in Example \ref{Zn}, Example \ref{G(n,1,3)}, and Example \ref{H}.  It follows that any arrangement which is lattice equivalent to the one of the arrangements defined by $[x^n-y^n][x^n-z^n][y^n-z^n]$, $[x^n(y^n-z^n)][y^n(x^n-z^n)][z^n(x^n-y^n)]$, or the Hesse configuration is complete. Currently, these are the only known examples of complete multinets. In fact, we show that the family of examples given by the Fermat pencil are the only complete $3$-nets. 

\begin{theorem}
\label{main}
A complete $(3,n)$-net is projectively equivalent to the arrangement with defining polynomial $[x^n-y^n][x^n-z^n][y^n-z^n]$. 
\end{theorem}

\begin{proof}
By Proposition 3.3 of \cite{Ynet}, it suffices to show each block of a complete $(3,n)$-net is a pencil. Using Definition \ref{complete} and Proposition \ref{properties}, the Riemann-Hurwitz type formula (\ref{E2}) for a complete $(3,n)$-net becomes
\begin{eqnarray*}
\sum_{p\in \overline{\mathcal{X}}}(m_p-1) & = & 2|\mathcal{A}|-|\mathcal{X}|-3(n+1)+\sum_{p\in \mathcal{X}}n_p \\
& = & 2(3n) -|\mathcal{X}| - 3(n+1)+|\mathcal{X}| \\
& = & 3(n-1).
\end{eqnarray*}
Consider the block $\mathcal{A}_i$. Note $n>1$ and select a line $\ell_0\in \mathcal{A}_i$. Since $\ell\cap \ell_0 \in \overline{\mathcal{X}}\cap P_{\mathcal{A}_i}$ for each $\ell \in \mathcal{A}_i\setminus\{\ell_0\}$, we have 
\begin{eqnarray*}
\sum_{p\in \overline{\mathcal{X}}\cap\ell_0}(m_p-1) & = & \left(\sum_{p\in \overline{\mathcal{X}}\cap\ell_0}m_p\right) -|\overline{\mathcal{X}}\cap\ell_0| \\
& = & [(n-1)+|\overline{\mathcal{X}}\cap\ell_0|]-|\overline{\mathcal{X}}\cap\ell_0| \\
& = & n-1.
\end{eqnarray*}
Moreover, $\overline{\mathcal{X}}\cap \ell_0 \subseteq \overline{\mathcal{X}}\cap P_{\mathcal{A}_i}$  and $m_p\geq 2$ for each $p\in \overline{\mathcal{X}}\cap P_{\mathcal{A}_i}$, hence 
\begin{eqnarray}
\label{E3}
\sum_{p\in \overline{\mathcal{X}}\cap P_{\mathcal{A}_i}}(m_p-1) & \geq & \sum_{p\in \overline{\mathcal{X}}\cap\ell_0}(m_p-1)=n-1.
\end{eqnarray}
It follows that 
\begin{eqnarray*}
\sum_{p\in \overline{\mathcal{X}}}(m_p-1) & \geq & 3 \cdot \min_i\left(\sum_{p\in \overline{\mathcal{X}}\cap P_{\mathcal{A}_i}}(m_p-1) \right) \\
& \geq & 3 (n-1).
\end{eqnarray*}
If $\overline{\mathcal{X}}\cap \ell_0 = \overline{\mathcal{X}}\cap P_{\mathcal{A}_i}$ for each $i$ and every $\ell_0\in \mathcal{A}_i$, or equivalently $\mathcal{A}_i$ is a pencil for each $i$, then equality holds and the $3$-net is complete. Conversely, suppose the $3$-net has another intersection point $p_0\in (\overline{\mathcal{X}}\cap P_{\mathcal{A}_i})\setminus (\overline{\mathcal{X}} \cap \ell_0)$ for some $i$. Then (\ref{E3}) is a strict inequality since the lefthand sum is larger by at least one and implies that the $3$-net is not complete. The result now follows.
\end{proof}

Vall\`es shows in Theorem 2.7 of \cite{V} that the union of all the singular fibers of a pencil of degree $n$ plane curves whose base locus consisting of $n^2$ distinct points is a free divisor. Since all singular fibers of a Ceva pencil are completely reducible for a complete multinet, we obtain the following statement.
\begin{theorem}
\label{free}
Complete nets are free.
\end{theorem}
Alternatively, this can be established for $3$-nets by combining Theorem \ref{main} and Theorem 6.60 from \cite{OT}. Note that the Hesse configuration is the only currently known $4$-net. It is complete (see \cite{FY}) and free (see Proposition 6.85 of \cite{OT}). 

\subsection{Completeness of Induced Multinets from $Q_n$} 
The known infinite families of complete $3$-multinets are related to induced multinets from $Q_n$. The $(3,2n)$-multinets of type $G(n,1,3)$ are induced multinets from $Q_n$ by choosing $H$ to be the plane $x_0=0$. Furthermore the $(3,n)$-nets from the Fermat pencil appear as their subarrarangements. The latter family of nets are not inducible directly from $Q_n$ with one exception, namely any $(3,2)$-net is complete and can be induced from $Q_1$. 

With induced multinets from $Q_n$ providing numerous examples of multinets, we investigate this class of multinets for completeness. Let $\mathcal{A}_p=\{\ell\in \mathcal{A}:p\in \ell\}$ denote the lines of $\mathcal{A}$ passing through the point $p$. A useful tool is the following local test  for completeness presented in \cite{FY}.

\begin{proposition}
\label{local}
Suppose $\mathcal{A}$ is a complete multinet. Then, for each $p\in \mathcal{X}$, \begin{eqnarray}
\label{E4}
2n_p-2 & = & \sum_{\ell \in \mathcal{A}_p}(m(\ell)-1). 
\end{eqnarray} 
In particular, if the multinet is light, then the multinet is complete only if it is a net. 
\end{proposition}

Note that (\ref{E4}) holds for each $p\in \mathcal{X}$ of a net. Also since the only proper $k$-multinets occur when $k=3$ (see Proposition \ref{properties}), we focus our attention on this situation. Using combinatorial properties of multinets, Proposition \ref{local} can be freshly reformulated to give a local test for completeness which is convenient to implement. 
\begin{corollary}
\label{ctest}
Suppose $\mathcal{A}$ is a complete $3$-multinet. Then, for each $p
\in \mathcal{X}$, 
\begin{eqnarray}
\label{E5}
|\mathcal{A}_p| & = & n_p+2.
\end{eqnarray}
\end{corollary}

\begin{proof}
For $3$-multinets, observe $$\sum_{\ell\in \mathcal{A}_p}(m(\ell)-1)=3n_p-|\mathcal{A}_p|.$$ Substituting into (\ref{E4}) and simplifying gives the statement. 
\end{proof}

\begin{theorem}
The only complete multinets induced from $Q_n$ are the $(3,2)$-net of Coxeter type $A_3$ and the $(3,2n)$-multinets of type $G(n,1,3)$. 
\end{theorem}

\begin{proof}
A complete description of multinets induced from $Q_n$ was given in Theorem \ref{classify} and classified into ten types. We refer to a specific type based on the numbering conventions used there.

It follows from Proposition \ref{local}, Theorem \ref{classify}, and Theorem \ref{main} that any light induced multinet from $Q_n$ is not complete with one exception. The $(3,2)$-net of Coxeter type $A_3$ can be induced from $Q_1$. It remains to investigate the completeness of the heavy induced multinets.

Induced multinets of type 1 which realize $G(n,1,3)$ are complete. Also the induced multinet of type $3$ from $Q_2$ realizes $G(2,1,3)$, hence is also complete. We use Corollary \ref{ctest} to show the remaining types are not complete by exhibiting a point $p\in \mathcal{A}$ where (\ref{E5}) does not hold. For type 2, choose $p\in \mathcal{X}$ with multiplicity $n$ and observe $|\mathcal{A}_p|=2n+1$. Next consider an induced multinet of type 3 (with $n>2$), of type 4, or of type 5. Choose $p\in \mathcal{X}$ to be a double point which lies on exactly one of the lines of multiplicity 2. Then $n_p=2$ and $|\mathcal{A}_p|=5$. This completes the proof.
\end{proof}

\section{Open Problems}
\label{sec:5}

We list several open problems. \\

{\bf Problem 1.} Are there examples of complete multinets other than the ones exhibited in Examples \ref{Zn}, \ref{G(n,1,3)}, and \ref{H}? \\

{\bf Problem 2.} How many specializations are possible for the nets realizing Latin squares of small order such as $\mathbb{Z}/4\mathbb{Z}$, $\mathbb{Z}/2\mathbb{Z}\times\mathbb{Z}/2\mathbb{Z}$, $\mathbb{Z}/5\mathbb{Z}$, and Latin square of order 5 which is not isotopic to the multiplication table of a group (see \cite{St1})? These are nets constructed by Stipins. \\

{\bf Problem 3.} Are there any general properties regarding the number of specializations of a given multinet? \\

{\bf Problem 4.} (Yuzvinsky) Is the Hesse configuration the unique $4$-net up to projective isomorphism?

\section*{Acknowledgements}
Thank you to Max Wakefield and the anonymous referee for their helpful comments for improving this paper, especially pointing out connections with the results of Vall\`es. Also, thank you to the organizers of \emph{Configuration Spaces} and \emph{Perspectives in Lie Theory}, Centro de Giorgi, Scuola Normale Superiore di Pisa, Istituto Nazionale di Alta Matematica, and many others involved for being wonderful hosts while working on this project during my visits to Italy.

\end{document}